\documentclass[a4paper]{amsart}
\usepackage{amsmath,amsthm,amssymb,latexsym,epic,bbm,comment,mathrsfs}
\usepackage{graphicx,enumerate,stmaryrd, xcolor, color, cancel}
\usepackage[all,2cell]{xy}
\xyoption{2cell}

\usepackage{soul}

\usepackage{tikz-cd}

\theoremstyle{plain}
\newtheorem{thm}{Theorem}
\newtheorem*{thm*}{Theorem}
\newtheorem*{thmA}{Theorem A}
\newtheorem*{thmB}{Theorem B}
\newtheorem*{thmC}{Theorem C}

\newtheorem{lem}[thm]{Lemma}
\newtheorem{prop}[thm]{Proposition}

\newtheorem{cor}[thm]{Corollary}
\newtheorem{df-prop}[thm]{Definition-Proposition}

\theoremstyle{definition}

\theoremstyle{remark}
\newtheorem{rem}[thm]{Remark}
\newtheorem{ex}[thm]{Example}

\usepackage[all]{xy}
\usepackage[active]{srcltx}
\usepackage[parfill]{parskip}
\usepackage{enumerate}

\usepackage{hyperref}
\newcommand{\mc}{\mathcal}
\newcommand{\mf}{\mathfrak}
\newcommand{\C}{\mathbb C}

\newcommand{\oa}{{\bar 0}}
\newcommand{\ob}{{\bar 1}}

\newcommand{\g}{\mathfrak{g}}
\def\Ann{{\text{Ann}}}

\def\ov{\overline}

\newcommand{\h}{\mathfrak{h}}

\newcommand{\Z}{{\mathbb Z}}

\def\sMod{\operatorname{-sMod}\nolimits}
\def\Wmod{\operatorname{-Wmod}\nolimits}

\def\Res{\operatorname{Res}\nolimits}
\def\Ind{\operatorname{Ind}\nolimits}

\def\spann{\operatorname{span}\nolimits}

\def\gr{\operatorname{gr}}

\def\ov{\overline}
\newcommand{\ad}{\mathrm{ad}}

\begin{document}
\title{Categorical Equivalences of Finite W-Superalgebras and Clifford Twists}

\author[Chen]{Chih-Whi Chen} \address{Department of Mathematics, National Central University, Chung-Li, Taiwan 32054 \\ National Center of Theoretical Sciences,
	Taipei, Taiwan 10617} \email{cwchen@math.ncu.edu.tw}

\author[Cheng]{Shun-Jen Cheng}
\address{Institute of Mathematics, Academia Sinica, Taipei, Taiwan 10617} \email{chengsj@as.edu.tw}

\author[Suh]{Uhi Rinn Suh} \address{Department of Mathematical Sciences and Research institute of Mathematics, Seoul National University,
	Gwanak-ro 1, Gwanak-gu, Seoul 08826, Korea} \email{uhrisu1@snu.ac.kr}

\begin{abstract}  
Associated with an even nilpotent element $e$ in a basic classical Lie superalgebra $\g$, we study, in full generality, two constructions of finite $W$-superalgebras, defined via Whittaker models and isotropic subspaces, respectively. We prove that both formulations are independent of the various choices made in their constructions, thereby yielding, for a fixed good grading for $e$, at most two isomorphism classes of $W$-superalgebras. In the case when there are two non-isomorphic versions, we establish that they differ precisely by a Clifford extension. Consequently, when the two $W$-superalgebras are non-isomorphic, their module categories are equivalent up to a Clifford twist. Building on this equivalence and utilizing the Skryabin equivalence,  we classify their irreducible representations in terms of generalized Whittaker modules over $\g$.
\end{abstract}

\maketitle




\section{Introduction}

\subsection{Background}	The finite $W$-algebra $U(\g,e)$  is an associative algebra   constructed from a nilpotent element $e$ in a finite-dimensional semisimple Lie algebra $\g$, which naturally generalizes the universal enveloping algebra $U(\g) = U(\g,0)$. Stemming from Kostant's seminal paper \cite{Ko78}, the representation theory of these objects has enjoyed a surge of interest in recent years, partly owing to works that made connections with other branches of representation theory; see, e.g., \cite{BT93, RS99, Pr02, BK06, Lo10b, LPTTW25} for a sample.

The general construction of a finite $W$-algebra, as formulated by Gan and Ginzburg \cite{GG02}, a priori, depends on a choice of an isotropic subspace $\mathfrak{l}$ of the symplectic space $\g(-1)$ associated with a good grading  $\g=\bigoplus_{k\in \mathbb{Z}}\g(k)$ for the nilpotent element $e$ of a semisimple Lie algebra \cite{EK05}. Accordingly, we shall denote it by $\mc W_{\mf l}$.  The most prominent cases occur at the two extremes, i.e., when $\mf l$ is either trivial or maximal isotropic. 
The first case, appearing in the context of mathematical physics in its affine counterpart, has received considerable attention in a slightly different incarnation; see, e.g., \cite{BT93,RS99,DK06}.  Meanwhile, the latter case, i.e., when $\mf l$ is maximal isotropic, as originally introduced by Premet \cite{Pr02}, has proven essential for investigating connections to the representation theory of Lie algebras; e.g., {\em Skryabin's equivalence} \cite{Skr02} establishes a categorical equivalence between the category of $\mc W_{\mf l}$-modules and that of the so-called Whittaker $\g$-modules. Gan and Ginzburg in \cite{GG02} prove the important result that the definition of a finite $W$-algebra is independent of the choice of the isotropic subspace $\mf l$, up to isomorphism.

The construction of finite $W$-algebras affords a natural generalization to the case when $\g$ is a Lie superalgebra. Parallel to the Lie algebra case, the finite $W$-algebras associated to Lie superalgebras under the formulation corresponding to the case $\mathfrak{l}=0$ arise naturally as {\em Zhu algebras} in the theory of vertex algebras, which provides a bridge connecting them to the richer theory of their affine counterparts. More precisely, since the correspondence between irreducible positive-energy modules of a vertex algebra and irreducible modules of the corresponding Zhu algebra is well understood, understanding the representation theory of finite $W$-algebras $\mc W_{\mf l=0}$ has important implications for the theory of affine W-algebras; see \cite{Zhu96, KRW03, KW04,DDCDS+06}.

On the other hand, the super-analogue of Premet's formulation of finite $W$-algebras, denoted by $U(\g,e)$, has also been extensively developed; see, e.g., \cite{BB13, Zh14, ZS15, PS16, SX20}. In this formulation, a super analogue of Skryabin's equivalence holds, implying that $U(\g,e)$-modules can be studied via the representation theory of Whittaker $\g$-modules.  In this paper, we compare the associative algebras $\mathcal W_{\mathfrak l}$ and $U(\g,e)$, together with their representation theories when $\g$ is a basic Lie superalgebra.

\subsection{Description of results} \label{sect::DR}
\subsubsection{}
Let $\g=\g_\oa\oplus\g_\ob$ be a basic Lie superalgebra equipped with a supersymmetric even invariant bilinear form $(\_ |\_ )$, and let $e\in \g_{\bar 0}$ be an even nilpotent element.    Consider a good grading $\Gamma$ for $e$: $\g=\bigoplus_{k\in \mathbb{Z}} \g(k)$, see \cite{EK05}. Define the subalgebra $
\mathfrak m_{\mathfrak l}=\mathfrak l\oplus\bigoplus_{k\le -2}\g(k),
$
for  $\mathfrak l\subseteq \g(-1)$ an isotropic subspace with respect to the super-symplectic bilinear form 
$(e| [\cdot,\cdot]).$

The natural analogue of Gan--Ginzburg's definition of a finite $W$-algebra is the finite $W$-superalgebra
\begin{equation} \label{eq:GG_def}
   \mathcal{W}_{\mathfrak l}=\left(U(\g)\otimes_{U(\mathfrak m_{\mathfrak l})}\mathbb C_\chi\right)^{\text{ad}\mathfrak m_{\mathfrak l}'}, 
\end{equation}
where $\chi:\mathfrak m_{\mathfrak l}\to \mathbb{C}$ is the character defined by $m\mapsto (m|e)$, for $m\in\mf m_{\mf l}$, and $\mathfrak m_{\mathfrak l}'=\mathfrak l'\oplus\bigoplus_{k\le -2}\g(k)$, where $\mf l'$ is the super-symplectic complement of $\mf l$ in $\g(-1)$. In \cite{GG02}, it was proved that, when $\g$ is a Lie algebra, the associative algebra structure of $\mathcal{W}_{\mathfrak l}$ does not depend on the choice of the isotropic subspace $\mathfrak l$. Later, Zhao extended this result to the setting of Lie superalgebras, provided that  $\rm r:=\dim\g(-1)$ is even, see \cite[Section 3, Remark 3.11]{Zh14}. 

The first theorem of this paper shows that the above result remains valid without the assumption on the parity of $\rm r$.

\begin{thmA}[Theorem \ref{thm::GGrev}]
    The associative algebra $\mathcal{W}_{\mathfrak l}$ in \eqref{eq:GG_def} is independent of the choice of the isotropic subspace $\mathfrak l$. Furthermore, if $\Gamma$ and $\Gamma'$ are two Dynkin (hence good) gradings for $e$, then their corresponding finite $W$-superalgebras are isomorphic.
\end{thmA}

Theorem~A implies that the algebra $\mc W_{\mf l}$ is independent of $\mf l$, and hence we shall freely use the notation $\mc W$ to denote it in the remainder of this section.

On the other hand, the analogue of Premet's definition of a finite $W$-algebra is 
\begin{equation}\label{eq:premet_def}
    U(\g,e)=\left(U(\g)\otimes_{U(\mathfrak m_{\bar{\mathfrak l}})}\mathbb C_\chi\right)^{\text{ad}\mathfrak m_{\bar{\mathfrak l}}}.
\end{equation}
Here, we denote by $\bar{\mathfrak l}$ a Lagrangian subspace of $\g(-1).$ Note that $\mathfrak m_{\bar{\mathfrak l}}=\mathfrak m_{\bar{\mathfrak l}}'$ when  $\dim\g(-1)={\rm r}$ is even. In this case, the construction reduces to the Gan--Ginzburg construction. However, this is no longer true when $\dim\g(-1)$ is odd, and it is not immediate that $U(\g,e)$ does not depend on the choice of $\bar{\mf l}$, see, e.g., \cite[Remark~3.11]{Zh14}. In this case, we show that there exists an odd element $\theta \in \bar{\mathfrak l}'\setminus \bar{\mathfrak l}$ such that
$
\mathfrak m_{\bar{\mathfrak l}}'=\mathfrak m_{\bar{\mathfrak l}}\oplus \mathbb{C}\theta.
$
One observes that the element $\theta \otimes 1 \in U(\g,e)$ generates a Clifford superalgebra, and it does not belong to $\mc W_{\ov{\mf l}} \cong\mc W$. Moreover, we prove the following theorem.

\begin{thmB}[Theorems \ref{thm::1} and \ref{thm::6}, Corollary \ref{cor::NonIso}]\label{thm:intro 2}
Suppose that $\rm r$ is odd. Then 
    \[U(\g,e) \simeq \mc W \otimes {\mc Cl}(1),\]
where ${\mc Cl}(1)$ is the Clifford algebra generated by one element. As a consequence, $U(\g,e)$ is independent of the choice of the Lagrangian subspace $\bar{\mathfrak l}$. Moreover, a comparison of the representation categories of $U(\g,e)$ and  $\mc W$ shows that the two algebras are not isomorphic.
\end{thmB}

   \subsubsection{} \label{sect::introThmC}
	The second part of this paper focuses on the representation theories of $U(\g,e)$ and $\mc W$.
     In particular, as a consequence of Theorem B, we deduce one of our main results, which both describes the relationship between the categories of $U(\g,e)$- and $\mc W$-modules and establishes an explicit bijection between irreducible $\mc W$-modules and generalized Whittaker modules over $\g$. To explain this part further, we provide more details on our setup below. In this paper, for a given associative superalgebra $A$, we let $A\text{-sMod}$ be the category of $A$-modules with even $A$-linear homomorphisms. Equipped with the parity-change functor $\Pi: A\text{-sMod} \to A\text{-sMod}$, the pair $(A\text{-sMod}, \Pi)$ naturally forms a {\em supercategory}. Corresponding to the supercategory $(A\text{-sMod}, \Pi)$, we can form a supercategory $(A\text{-sMod}^{\texttt{CT}}, \Pi^{\texttt{CT}})$ called its Clifford twist \cite[Section 2.1]{KKT16}; see Section \ref{sect::Cltw}.
    
We say that an irreducible $A$-module $X\in A\text{-sMod}$ is of {\em type $\texttt{Q}$} if $X$ and its parity-reversed module $\Pi X$ are isomorphic in $A\text{-sMod}$, and of {\em type $\texttt{M}$} otherwise. For any full subcategory $\mc C$ of $A\text{-sMod}$, we denote by $\mathrm{Irr}^{\texttt{M}}(\mc C)$ and $\mathrm{Irr}^{\texttt{Q}}(\mc C)$ the sets of isomorphism classes of simple objects of type $\texttt{M}$ and type $\texttt{Q}$ in $\mc C$, respectively.

Suppose that the dimension of $\g(-1)$ is odd. By virtue of the isomorphism $U(\g,e) \cong \mc W \otimes {\mc Cl}(1)$ from Theorem~B, and in analogy with the functors introduced in \cite{BK02,W09}, we consider the restriction functor  
$G(-)\colon U(\mathfrak{g},e)\text{-sMod} \to \mc W\text{-sMod}$, 
which assigns to each $U(\mathfrak{g},e)$-module $M$ the same underlying superspace with the action restricted to $\mc W$. 
The functor $G$ admits a left adjoint given by the induction functor
\begin{equation*}
    F(-) = U(\g,e)\otimes_{\mc W}-\colon \mc{W}\text{-sMod} \to U(\mf{g},e)\text{-sMod}.
\end{equation*}

           The following is our third main result:
      
        \begin{thmC}[Theorem \ref{thm::UWsimples}] \label{thm::WUgemodule}
        	There is an equivalence of supercategories  $$\mathbb E \colon (U(\g,e)\emph{-sMod}, \Pi) \to (\mc W\emph{-sMod}^{\textnormal{\texttt{CT}}},\Pi^{\textnormal{\texttt{CT}}}).$$          Furthermore,  $F$ and $G$ are superfunctors between $(\mc{W}\emph{-sMod}, \Pi)$  and  $(U(\g,e)\emph{-sMod}, \Pi)$,  
        	inducing two-to-one maps $$\emph{Irr}^{\textnormal{\texttt{M}}}(\mc{W}\emph{-sMod})\rightarrow \emph{Irr}^{\textnormal{\texttt{Q}}}(U(\g,e)\emph{-sMod})$$ and $$\emph{Irr}^{\textnormal{\texttt{M}}}(U(\g,e)\emph{-sMod})\rightarrow \emph{Irr}^{\textnormal{\texttt{Q}}}(\mc{W}\emph{-sMod}),$$ respectively.
	\end{thmC}

    \subsubsection{} \label{sect::123}
        Recall that a $\g$-module $M$ is a \emph{generalized Whittaker module} associated with $\chi$ if $x - \chi(x)$ acts locally nilpotently on $M$, for all $x \in \mf m$. Let $\g\text{-}\mathrm{Wmod}^{\chi}$ denote the category of all such modules. Recall further that the celebrated \emph{Skryabin equivalence} asserts that the \emph{Whittaker functor} provides an equivalence of supercategories
\[
    \operatorname{Wh}_\chi(-) \colon (\g\text{-}\mathrm{Wmod}^{\chi}, \Pi) \xrightarrow{\;\simeq\;} (U(\g,e)\text{-}\mathrm{sMod}, \Pi), 
\]
where for any generalized Whittaker module $M \in \g\text{-}\mathrm{Wmod}^{\chi}$. Recall that the space
\[
    \operatorname{Wh}_\chi(M) := \{ m \in M \mid x \cdot m = \chi(x)m \text{ for all } x \in \mf m \}
\]
is equipped with the induced $U(\g,e)$-action. Such an equivalence was originally established by Skryabin \cite{Skr02} in the context of Lie algebras, and was later generalized to the setting of Lie superalgebras in \cite{Zh14, ZS15, SX20}; see also Theorem~\ref{thm::Skr}.  Combining this with Theorem~C, we arrive at a classification of irreducible $\mc{W}$-modules:         
	\begin{cor}[Corollary \ref{thm::UWsimples}] \label{cor::1st}
			For each irreducible generalized Whittaker module $M \in \g\Wmod^\chi$, let $\widetilde{G}(M)$ be an irreducible submodule of $G(\operatorname{Wh}_\chi(M))$. Then, the sets \begin{align}\label{eq::introcoro}
            &\{\widetilde{G}(M)\mid M\in \emph{\text{Irr}}^{\textnormal{\texttt{M}}}(\g\Wmod^\chi)\}~\text{ and }~  \{\widetilde{G}(M)\mid M\in \emph{\text{Irr}}^{\textnormal{\texttt{Q}}}(\g\Wmod^\chi)\}
			\end{align}
			provide exhaustive lists of all irreducible $\mc W$-modules of type $\textnormal{\texttt Q}$ and type $\textnormal{\texttt M}$, respectively, up to parity change. 
	\end{cor}

The final piece of motivation for the present paper is to study the finite-dimensional  modules over finite $W$-superalgebras $U(\g,e)$ and $\mc{W}$. In Proposition \ref{prop::FrobeniusExt}, we construct a natural adjoint pair $(\mathsf{Ind}, \mathsf{Res})$ of exact functors between the categories of finite-dimensional modules over $U(\g_\oa,e)$ and $U(\g,e)$, providing a categorical bridge between finite $W$-algebras and finite $W$-superalgebras.
In the special case $e=0$, the adjoint pair $(\mathsf{Ind}, \mathsf{Res})$ recovers the classical induction and restriction functors induced by the $(\g, \g_\oa)$-bimodule $U(\g)$.
Using this   framework, we prove the existence of finite-dimensional $U(\g,e)$- and $\mc{W}$-modules, and obtain a bijection between their simple modules; see Proposition  \ref{res::fdim}.

 \subsection{Organization} 
 This paper is organized as follows. In Section \ref{sect::finWS}, we first introduce basic notation and review constructions for the finite $W$-superalgebras $\mc{W}$ and $U(\mathfrak{g},e)$. We then establish Theorem A in Section \ref{sect::1}, showing the independence of $\mc{W}$ of the choice of the Lagrangian subspace $\mathfrak l$, and further prove in Theorem B in Section \ref{sect::PrDef} an isomorphism $U(\g,e) \cong \mc{W} \otimes {\mc Cl}(1)$. In Section \ref{sect::repW}, we study representations of $\mc{W}$ and $U(\g,e)$, as well as the generalized Whittaker modules over $\g$. 
 Based on the Skryabin equivalence, which we recall in Section \ref{sect::Skreq}, we establish in Section \ref{sect::322} the existence of an adjoint pair of exact functors between $U(\g,e)\sMod
 $ and $U(\g_\oa,e)\sMod$. In Section~\ref{sect::IrrWell}, we introduce the Clifford twist, complete the proof of Theorem C and establish that $U(\g,e)$ and $\mc{W}$ are not isomorphic, thereby completing the proof of Theorem B. In Section~\ref{sect::eg}, we provide several applications and detailed examples. In Section~\ref{sect::appendix}, we characterize finite-dimensional $U(\g,e)$-modules by exploiting  properties of the annihilator ideals of their images under the Skryabin equivalence. In Section~\ref{sect::341}, we show that the ghost center of $U(\g)$ embeds into the finite $W$-superalgebra $\mathcal{W}$. Finally, Section~\ref{sec:exam:osp} is devoted to a detailed analysis of the principal finite $W$-superalgebra of $\mf{osp}(1|2n)$.
  \vskip0.2cm
 {\bf Acknowledgments}. The first two authors are   partially supported by National Science and Technology Council grants of the R.O.C., and they further acknowledge support from the National Center for Theoretical Sciences. The third author is supported by National Research Foundation of Korea (NRF) Grant No. 2022R1C1C1008698.

\section{Finite $W$-superalgebras} \label{sect::finWS}
\subsection{The setup} Throughout this paper, all vector spaces and algebras are assumed to be over the complex field $\C$ and we let $\Z_2 = \{\bar{0}, \bar{1}\}$ denote the cyclic group of order two.  For a homogeneous element $v$ in a superspace, i.e., a $\Z_2$-graded vector space, we denote
its parity by $|v|\in \Z_2$. 

We fix   $\g=\g_\oa\oplus \g_\ob$  a {\em basic} Lie superalgebra from Kac’s list \cite{K77}: 
\begin{align}
&\mf{gl}(m|n), ~~ \mf{sl}(m|n), ~~ \mf{psl}(n|n), ~~ \mf{osp}(m|2n), ~~ D(2,1; \alpha), ~~ G(3), ~~F(4). \label{eq::Kaclist}
\end{align}
It follows that $\g_\oa$ is a reductive Lie algebra, and $\g_\ob$ is completely reducible as a $\g_\oa$-module under the adjoint action. Furthermore, $\g$ admits a non-degenerate $\g$-invariant supersymmetric bilinear form $(\_|\_)$. Fix an even non-zero nilpotent element $e \in \g_\oa$. Recall from \cite{EK05} that a {\em good grading} for $e$ is a $\Z$-grading $$\Gamma:~\mf{g} = \bigoplus_{k \in \Z} \mf{g}(k)$$ satisfying  $e\in \g(2)$, such that the adjoint action $\ad e: \g(k) \rightarrow \g(k+2)$ is injective for $k\leq -1$ and surjective for $k \geq-1$; see also \cite{H12} for a classification of good gradings for basic Lie superalgebras. 
 For any such a grading $\Gamma$, there exists $h\in \g(0)$ and $f\in\g(-2)$ such that $\{e, h, f\}\subset \g_{\bar{0}}$ form an $\mf{sl}(2)$-triple (see, e.g., \cite[Lemma~1.1]{EK05}, \cite[Lemma 25]{W11}). Such an $\mf{sl}(2)$-triple is referred to as a $\Gamma$-graded $\mf{sl}(2)$-triple. In particular, if $\Gamma$ coincides with the eigenspace decomposition of $\ad h$, then $\Gamma$ is called a {\em Dynkin grading} for $e$. While the Dynkin grading always exists by the Jacobson–Morozov theorem, there may exist other good gradings that are not Dynkin.

\label{sect::FiniteW}
\subsection{Gan--Ginzburg's isotropic subspace definition} \label{sect::1} 

Let $\Gamma:~\mf{g} = \bigoplus_{k \in \Z} \mf{g}(k)$  be a good grading for an even non-zero nilpotent element $e \in \g_\oa$.  Let $\chi(\_): = (e|\_): \g\rightarrow \C$. Then the subspace $\mf g(-1)$ is equipped with a super-symplectic form: $$\omega_\chi(v,w):= \chi([v,w]), ~\text{ for }v,w \in \g(-1).$$ Let $\mf l \subseteq \g(-1)$ be an isotropic subspace and $\mf l'\subseteq \mf g(-1)$ be the corresponding subspace orthogonal to $\mf l$ with respect to $\omega_\chi$. Define the nilpotent subalgebras $$\mf m = \mf m_{\mf l}: = \mf l \oplus \bigoplus_{k\leq -2}\g(k)  ~~~ \text{  and  } ~~~ \mf m'= \mf m_{\mf l}': = \mf l' \oplus \bigoplus_{k\leq -2}\g(k).$$ 
Denote by $U(\g)$ the universal enveloping algebra of $\g$. Let $Q_{\mf l} = U(\g)/ I_{\mf l},$ where $I_{\mf l}$ is the left ideal of $U(\g)$ generated by $a-\chi(a)$, for $a\in \mf m_{\mf l}$. We view $Q_{\mf l}$ as an $\mf m'$-module via the adjoint action of $\mf m'$. 
The construction of finite $W$-algebras in the sense of Gan–Ginzburg in \cite{GG02} admits a natural generalization to the case of Lie superalgebras (see also \cite{KRW03, KW04, DK06, DDCDS+06}):
\begin{align}
&\mc W_{\mf l}:=Q_{\mf l}^{\ad \mf m'} \equiv \{y+I_{\mf l}\mid [a,y] \in I_{\mf l},\text{ for }a\in \mf m'\}. 
\end{align} Then $\mc W_{\mf l}$ has the structure of an associative superalgebra inherited from that of $U(\g)$. The subalgebra $\mf m$  appearing in the definition of the finite $W$-superalgebra ${\mc W_{\mf l}}$ depends, a priori, on the choice of the isotropic subspace $\mf l$. Nevertheless, Theorem \ref{thm::GGrev} shows that the different choices of $\mf l$ yield isomorphic finite $W$-superalgebras.

Let $0= U_0(\g)  \subset U_1(\g) \subset U_2(\g)\subset \cdots$ be the standard PBW filtration on $U(\g)$. The adjoint action $\ad h$  induces a grading on each $U_n(\g)$ by 
\begin{align*}
&U_n(\g)(i) =\{x\in U_n(\g)\mid \ad h(x) = i x\}. 
\end{align*}  The Kazhdan filtration $\cdots \subset F_n U(\g) \subset F_{n+1} U(\g)\subset \cdots $ on $U(\g)$ is a $\Z$-filtration defined by letting 
\begin{align*}
&F_k U(\g) = \sum_{i+2j\leq k} U_j(\g)(i),~\text{ for }k\in \Z.
\end{align*} 
Denote the corresponding associated graded by $\text{gr}^KU(\g)$. Note that the filtration $F_kU(\g)$ induces a Kazhdan filtration $F_k Q_{\mf l}$ of $Q_{\mf l}$, and, furthermore, each $F_kU(\g)$ (and hence each $F_kQ_{\mf l}$) is $\text{ad}\mf m'$-invariant.
   Let $\kappa: \g \xrightarrow{\cong} \g^\ast$ be the isomorphism induced by the non-degenerate $\g$-invariant  bilinear form $(\_|\_)$. 
Let $\{e,h,f\}\subset \g_\oa$ be a $\Gamma$-graded $\mf{sl}(2)$-triple. Following the terminology of Gan and Ginzburg \cite{GG02}, the {\em Slodowy slice} of $\g$ associated with $e$ is defined as 
\begin{align*}
&\mc S = \kappa(e+\g^f)=\chi +\ker(\ad^\ast f), 
\end{align*}where $\g^f = \{x\in \g\mid[f,x]=0\}$ and $\ad^\ast f$ denotes the coadjoint action of $f$ (on $\g^*$). Set $\mf m^{\ast, \perp} = \kappa(\mf m^\perp)$, where $\mf m^\perp$ is the orthogonal complement of $\mf m$ in $\g$ with respect to $(\_|\_)$. We have that $\g^f\subseteq\bigoplus_{j\le 0}\g(j)\subseteq \mf m^\perp$ so that $\mc S\subseteq \chi+\mf m^{\ast,\perp}$. There are canonical isomorphisms $\gr^K U(\g)\cong \C[\g^\ast]$ and $\gr^K Q_{\mf l}=\gr^K U(\g) /\gr^K I_{\mf l} \cong \C[\chi+\mf m^{\ast, \perp}]$. The Kazhdan filtration on $Q_{\mf l}$ induces the Kazhdan filtration on $\mc W_{\mf l}$ via the inclusion $\mc W_{\mf l}\hookrightarrow Q_{\mf l}$. This induces  an injective homomorphism of the associated graded algebras $\gr^K \mc W_{\mf l}\hookrightarrow \gr^K Q_{\mf l}$. Therefore, we have the following commutative diagram: \\
\begin{displaymath}
\begin{tikzcd}
\operatorname{gr}^K U(\mathfrak{g}) \arrow[d] \arrow[r, equal] & \mathbb{C}[\mathfrak{g}^\ast] \arrow[d] \\
\operatorname{gr}^K Q_{\mathfrak{l}} \arrow[r, equal] & \mathbb{C}[\chi + \mathfrak{m}^{\ast, \perp}] \arrow[d] \\
\operatorname{gr}^K \mathcal{W}_{\mathfrak{l}} \arrow[u] \arrow[r, "\nu"] & \mathbb{C}[\mc S]
\end{tikzcd}
\end{displaymath}

\noindent where $\C[\g^\ast] \rightarrow \C[\chi+\mf m^{\ast, \perp}]\rightarrow \C[\mc S]$ are the restriction homomorphisms induced by the inclusions $\g^\ast  \supset \chi+\mf m^{\ast,\perp} \supset  \mc S $, and $\nu: \operatorname{gr}^K \mathcal{W}_{\mathfrak{l}} \to \mathbb{C}[\mc S]$ is the composition of the  natural maps:
\begin{equation*}
\operatorname{gr}^K \mathcal{W}_{\mathfrak{l}} \to \operatorname{gr}^K Q_{\mathfrak{l}}\xrightarrow{\cong}  \mathbb{C}[\chi + \mathfrak{m}^{\ast, \perp}] \to \mathbb{C}[\mc S].
\end{equation*}

The following theorem is a super-analogue of \cite[Theorem 4.1, Propositions~5.1,~5.2]{GG02}: 
\begin{thm} \label{thm::GGrev}
 We have
\begin{itemize}
    \item[(i)] The map $\nu$ is an isomorphism of graded associative superalgebras.  
    \item[(ii)] $\gr^K H^n(\mf m', Q_{\mf l})= H^n(\mf m', \gr^K Q_{\mf l}) = 0$ for any $n>0$. 
       \item[(iii)] For a fixed good grading for $e$, the finite $W$-superalgebras $\mc W_{\mf l}$ are all isomorphic for different choices of isotropic subspace $\mf l$. 
        \item[(iv)] Different choices of Dynkin gradings for $e$ yield isomorphic  $\mc W_{\mf l}$.
\end{itemize}
\end{thm}
\begin{proof}  We first fix a good grading $\Gamma$ and a $\Gamma$-graded $\mf{sl}(2)$-triple $\{e,h,f\}$ for $e$.
A similar result was established in \cite[Theorems 3.5, 3.6]{Zh14} for queer Lie superalgebras. As noted in \cite[Remark 3.11]{Zh14}, those results were expected to extend to basic classical Lie superalgebras with even ${\rm r}:=\dim\g(-1)$. Indeed, the arguments in \cite{Zh14} can be naturally adapted to our current setting of basic classical Lie superalgebras to prove Theorem \ref{thm::GGrev}  for any ${\rm r}$.    
 One ingredient for the proof is the fact that the analogue of \cite[Lemma 3.3]{Zh14} remains valid in our setting, namely,  the coadjoint action map 
\begin{align} \label{eq:GGdecom}
&\alpha: M'\times \mc S\rightarrow \chi+\mf m^{\ast, \perp}
\end{align}is an isomorphism of affine superschemes, where $M'$ is the unipotent algebraic supergroup whose Lie superalgebra is $\mathfrak{m}'$. This follows by a direct extension of the arguments in \cite[Lemma 3.3]{Zh14}, as the original proof remains valid in our more general setting without  modification.  
This decomposition yields an isomorphism $\C[\chi+\mf m^{\ast, \perp}]\cong \C[M']\times \C[\mc S]$ of associative superalgebras and of $M'$-modules, where $M'$ acts on $\C[M']\times \C[\mc S]$ via the action induced by left translation on $M'$.  It follows that 
\begin{align*}
&H^n(\mf m', \gr^K Q_{\mf l}) = H^n(\mf m', \C[\chi+{\mf m^{\ast,\perp}}]) = H^n(\mf m', \C[M'])\otimes \C[\mc S], \text{ for }n\geq 0.
\end{align*}  In particular, we have $H^0(\mf m', \gr^K Q_{\mf l})=(\gr^K Q_{\mf l})^{\mf m'} = \C[\mc S]$.

Here, we sketch an alternative proof of the decomposition \eqref{eq:GGdecom}, extending the approach used in the proof of \cite[Proposition 18]{Na23} where the case $\mf l=0$ was considered. First, consider the adjoint action of the affine supergroup $M'$ on $e+\g^f$ given by $$\alpha': M'\times (e+\g^f) \rightarrow e+\mf m^{\perp},~(g,e+X)\mapsto g(e+X)g^{-1},$$ for $g\in M'$ and $X\in \g^f$. It suffices to show that $\alpha'$ is an isomorphism of all $A$-valued points for any commutative associative $\C$-superalgebra $A$. Let  $e+x\in e+\g^f(A)$ and $g\in  M'(A)$. Since $M'$ is a unipotent affine algebraic supergroup, it follows by \cite[Corollary 5]{Na23} that there is $y= y_x\in \mf m'(A)$ such that 
\begin{align*}
&g(e+x)g^{-1} = \sum_{n\geq 0} \frac{1}{n!}\ad(y)^n(e+x),
\end{align*} see also the proof of \cite[Proposition 18]{Na23}.

To show the bijectivity of $\alpha'$, we follow the strategy of \cite[Lemma 3.13]{H22} which we sketch below. For any $r\in \g(A)$ we let $r= \sum_{i\in \Z}r_i$, where $r_i\in \g(i)(A)$. Note that $\mf m^\perp = \bigoplus_{i\leq 0}\g(i)\bigoplus [e, \mf l']$. Therefore, for any $z \in \mf m^\perp(A)$, the equality $e+z= g(e+x)g^{-1}$ is equivalent to the following equalities: 
\begin{align*} 
&z_k - x_k - [y_{k-2}, e] \\ &= \sum_{i+j=k} \operatorname{ad} y_i(x_j) + \sum_{n \ge 2} \frac{\sum_{i_1 + \dots + i_n = k-2} \operatorname{ad} y_{i_1} \dots \operatorname{ad} y_{i_n}(e)}{n!} \nonumber \\ &+ \sum_{n \ge 2} \frac{\sum_{i_1 + \dots + i_n + j = k} \operatorname{ad} y_{i_1} \dots \operatorname{ad} y_{i_n}(x_j)}{n!},
\end{align*} for any $k\leq 1$. By following the decreasing induction on $k$ and utilizing the decomposition $\g(i) = [e, \g(i-2)] \oplus \g^f(i)$ for all $i$, as presented in \cite[Lemma~3.13]{H22}, one can establish that for any $z \in \mf m^{\perp}(A)$, there exists a unique solution $(x, y)$ to the equation $e+z= g(e+x)g^{-1}$.

The crucial fact that facilitates this adaptation is the vanishing of the Lie superalgebra cohomologies $H^n(\mf m', \C[M'])$, for all $n \ge 1$, where $\mathbb{C}[M']$ denotes the coordinate ring of the unipotent algebraic supergroup $M'$; see, e.g., \cite[Proposition 6]{Na23} for more general statement.  
This addresses the concern raised in \cite[Remark 3.11]{Zh14} regarding the extension to cases where ${\rm r}$ is odd. 
This vanishing result, combined with the Gan-Ginzburg type arguments employed in \cite[Lemma 3.3, Theorems~3.5,~3.6]{Zh14} (see also \cite[Sections 5.3-5.5]{GG02}), is sufficient to establish conclusions (i), (ii), as well as the independence of ${\mc W}_{\mf l}$ from the choice of the isotropic subspace $\mathfrak{l}$.

Finally, we show that the finite $W$-superalgebras $\mathcal{W}_{\mathfrak{l}}$ defined via Dynkin gradings are all isomorphic to one another. To see this, let $\Gamma$ and $\Gamma'$ be two Dynkin gradings induced by two $\mf{sl}(2)$-triples $$\langle e,h,f\rangle~\text{ and }~\langle e,h',f'\rangle \hookrightarrow \g.$$ Denote by $\mc W_{0}^\Gamma$ and $\mc W_{0}^{\Gamma '}$ the finite $W$-superalgebras associated with these gradings, where we take the isotropic subspace $\mf l=0$. It suffices to show that $\mc W_{0}^\Gamma\cong \mc W_{0}^{\Gamma '}$. We note that $\langle e,h,f\rangle~\text{ and }~\langle e,h',f'\rangle$ lie in the same orbit under the action of the adjoint group of $\g$ by a classical result of Kostant \cite[Theorem 3.6]{Ko59}; see also \cite[Theorem~3.4.10]{CM93} and \cite[Section~2.6]{W11}. Since each automorphism $\exp(\ad X)$ of $\g_\oa$, for $X\in \g_\oa$ nilpotent, extends to an automorphism of $\g$, it follows that there is an automorphism $\sigma$ of $U(\g)$ that maps $\langle e,h,f\rangle$ to $\langle e,h',f'\rangle$. Consequently, $\sigma$ descends to an isomorphism from $\mc W^\Gamma_{0}$ to $\mc W^{\Gamma'}_{0}$. This completes the proof.
\end{proof}
\begin{rem}
Part (iii) of Theorem \ref{thm::GGrev} recovers \cite[Theorem 4.11]{ZS19}, where the isomorphism $\mc W_{\mf l}\cong \mc W_0$ was established specifically for the case where $\mf l$ is Lagrangian and the grading is Dynkin.\end{rem}

In light of Theorem \ref{thm::GGrev} the finite $W$-superalgebra $\mc W_{\mf l}$ is independent of the choice of the subspace $\mf l\subseteq\g(-1)$ chosen. Hence, in the sequel, we shall drop the subscript and also use $\mc W$ to denote this algebra.

\subsection{Premet's Whittaker model definition} \label{sect::PrDef} Let $e \in \mf{g}_0$ be an even nilpotent element and $\Gamma: \mf{g} = \bigoplus_{k \in \Z} \mf{g}(k)$ be a good grading for $e$. In this subsection, we recall the definition of finite $W$-superalgebras via the Premet construction. This approach generalizes Premet's original framework in \cite{Pr02} for Lie algebras. In the special case when the isotropic subspace $\mf l$ is Lagrangian and the dimension of $\g(-1)$ is even, this construction recovers the Gan–Ginzburg-type definition introduced in Section \ref{sect::1}.

Fix a Lagrangian, i.e., a maximal isotropic, subspace $\mf l\subseteq \g(-1)$ with respect to the supersymplectic form $\omega_{\chi}$ on $\g(-1)$.  Set $${{\rm r}:=\dim \g(-1).}$$ 
 Here we remark that $\dim \g(-1)_{\bar{0}}$ is always even and hence the parity of $\rm r$ is the same as the parity of $\dim \g(-1)_{\bar{1}}$.
 
\begin{ex}  Suppose that $\g$ is a Lie superalgebra of type I, that is, $$\g =\mf{gl}(m|n), ~~ \mf{sl}(m|n), ~~ \mf{psl}(n|n),~\mf{osp}(2|2n),$$ equipped with a Dynkin grading $\Gamma$ induced by an $\mf{sl}(2)$-triple $\{e,h,f\}$. Then, as $\g_\oa$-modules under the adjoint action, the odd part $\g_\ob$ decomposes as $\g_\ob \cong V \oplus V^*$, where $V^*$ denotes the dual module of $V$. Upon restriction to the $\mf{sl}(2)$-subalgebra $\mf{s} := \langle e,h,f \rangle$, both $V$ and $V^*$ become isomorphic as $\mf{s}$-modules due to the self-duality of finite-dimensional $\mf{sl}(2)$-modules. Consequently, $\mf{g}_\ob$ restricts to a direct sum of two isomorphic $\mf{s}$-modules, which implies that the integer $\rm r$ is even. \end{ex}
 		
 		  \begin{ex} \label{ex::7} Suppose that $\g$ is a basic Lie superalgebra of type II, namely, $$\g  =  \mf{osp}(m|2n) \text{ with } m \neq 2,~D(2,1;\alpha),~G(3),~F(4),$$  and let $\Gamma$ be the Dynkin grading on $\g$ induced by  an $\mf{sl}(2)$-triple $\{e,h,f\}$.  Let $\g^e$ denote  the centralizer of $e$. We further assume that $e$ is principal in  $\g_\oa$ (i.e., $\dim \g^e_\oa$ is equal to the rank of $\g_\oa$). According to \cite[Lemma 2.9, Corollary 2.10]{PS16}, the integer $\rm r$ is odd if and only if $\g$ is one of the following Lie superalgebras:
 		  	\[
 		  	\g = \begin{cases} 
 		  		\mathfrak{osp}(2m+1|2n) & \text{with } m < n; \\ 
 		  		\mathfrak{osp}(2m|2n) & \text{with } n < m; \\ 
 		  		D(2,1; \alpha), \ F(4). & 
 		  	\end{cases}
 		  	\]
 \end{ex}
 
 In the case when ${\rm r}$ is odd, we pick a vector $\theta\in \mf l'\setminus \mf l $  with $\omega_\chi(\theta,\theta)=1$  so that we have
 \begin{align*}
 	&\mf m':=\left\{ \begin{array}{ll} \mf m, &  \text{~for ${\rm r}$ even;} \\
 		\mf m \oplus \C \theta, &  \text{~~for ${\rm r}$ odd.} \end{array} \right. \label{eq::extm}
 \end{align*}  
We  define the finite $W$-algebra $U(\g,e)$ in the spirit of Premet \cite{Pr02}:
 \begin{align*}
&U(\g,e): = Q_{\mf l}^{\ad \mf m}\cong \text{End}_{\g}(Q_{\mf l})^{\text{opp}}.
 \end{align*}  
Thus, $Q_{\mf l}$ is a $\left( U(\g),U(\g,e)\right)$-bimodule. If ${\rm r}$ is even, then  $\mf l = \mf l'$, and hence, by definition, $U(\g,e)$ coincides with $\mc W_\mf l$. In the case when ${\rm r}$ is odd, the finite $W$-superalgebra is also referred to as the {\em refined finite $W$-superalgebra} and is studied in \cite{ZS15, ZS19}. The following shows that $U(\g,e)$ and $\mc W_{\mf l}$ are {\em Morita super equivalent} in the sense of \cite{W09}. 

 \begin{thm} \label{thm::1}
 Suppose that ${\rm r}$ is odd and $\mf l\subseteq \mf g(-1)$ is a Lagrangian subspace with respect to $\omega_\chi$. Then we have 
 \begin{align}
 & U(\g,e) \cong \mc W_{\mf l}\otimes {\mc Cl}(1),
 \end{align} where ${\mc Cl}(1)$ is the Clifford superalgebra generated by an odd element $\Theta$ subject to $[\Theta,\Theta]=1$. 
 \end{thm}
 \begin{proof}  
 For each $x\in U(\g)$ and $a\in \mf m'$, we write $\ov x = x+I_{\mf l}$ and $[a,\ov x] = \ov{[a,x]}$. In what follows, we view $\mc W_{\mf l}$ as a subalgebra of $U(\g,e)$. Let $\ov 1$ be the identity of $U(\g,e)$. Note that $\ov \theta\in U(\g,e) \setminus \mc W_{\mf l}$, since $[\theta,\ov \theta] = \omega_\chi(\theta,\theta) \ov 1= \ov 1\in  Q_{\mf l}^{\ad \mf m}$ and $\ov \theta$ generates a subalgebra of $U(\g,e)$ isomorphic to ${\mc Cl}(1)$ via $\theta\mapsto \Theta$.  Define a linear  map $\psi: \mc W_{\mf l} \otimes {\mc Cl}(1) \rightarrow U(\g, e)$ by letting 
 \begin{align*}
&\psi\colon~\ov y\otimes 1\mapsto \ov y, ~\ov y\otimes\Theta\mapsto \ov{y} \ov{\theta},~\text{ for all } \ov y\in U(\g,e).
 \end{align*}Since $ [\ov \theta,\ov y]= [\theta,\ov y] =0$ in $U(\g,e)$, for any $\ov y\in W_{\mf l}$, it follows that $\psi$ is a well-defined superalgebra homomorphism.  We shall show that $\psi$ is an isomorphism. To see this, let $\ov y_1, \ov y_2 \in \mc W_{\mf l}$ be homogeneous elements such that $\ov y_1\otimes \Theta+ \ov y_2\otimes 1\in \mc W_{\mf l}\otimes{{\mc Cl}(1)}$ lies in the kernel of $\psi$. Then $\ov y_1 \ov \theta+\ov y_2 =0$ in $U(\g,e)$. We calculate the following  equalities of elements in $U(\g,e)$: 
 \begin{align*}
 0&=[\theta, \ov y_1 \ov \theta+\ov y_2] = [\theta, \ov y_1] \ov \theta+ (-1)^{|y_1|} \ov y_1 [\ov \theta,\ov \theta] +[\theta,\ov y_2] \\ &= (-1)^{|y_1|} {\ov y_1} [\ov \theta,\ov \theta] = (-1)^{|y_1|}\ov y_1, 
 \end{align*} since $[\theta,\ov y_1] = [\theta,\ov y_2]=0$ and $[\ov \theta,\ov{\theta}] =\ov 1$. This implies that $\ov y_1 = \ov y_2 =0$. Therefore, $\psi$ is injective. To show the surjectivity of $\psi$, we may note that 
 \begin{align*}
  &\ov x=[\ov{x}\ov \theta, \ov \theta]+[\ov x,\ov \theta]\ov \theta, ~\text{ for any }~\ov x\in U(\g,e).
 \end{align*} since $[\ov \theta, \ov \theta]=\ov 1$. By \cite[Proposition 4.9]{ZS15}, we have $[\ov{x}\ov \theta, \ov \theta],~[\ov x,\ov \theta]\in {\mc W_{\mf l}}$.
 Since $\psi(\mc W_{\mf l}\otimes 1) =\mc W_{\mf l}$, it follows that $\psi$ is surjective. Alternatively, we can prove that the inverse of $\psi$ is given by  the following linear map
 \begin{align*}
&\phi: U(\g,e)\rightarrow \mc W_{\mf l}\otimes {{\mc Cl}(1)}, ~\phi(\ov x):= [\ov x\ov \theta, \ov \theta]\otimes 1 +[\ov x,\ov \theta]\otimes \Theta, \text{ for any }\ov x \in U(\g,e).
 \end{align*} This completes the proof.
 \end{proof}

\begin{thm} \label{thm::6}
For a fixed good grading for $e$, the finite $W$-superalgebras $U(\g,e)$ are all isomorphic for different choices of the Lagrangian subspace $\mf l$. Furthermore, the isomorphism class of $U(\g,e)$ defined via Dynkin gradings remains independent of the specific choice of $\Gamma$. 

\end{thm}
\begin{proof}  Fix a good grading $\Gamma$ and a Lagrangian subspace $\mf l\subseteq \g(-1)$. By Part~(iii) of Theorem \ref{thm::GGrev}  and Theorem \ref{thm::1}, we have isomorphisms 
 $$U(\g,e) \cong \mc W_{\mf l}\otimes {\mc Cl}(1) \cong \mc W\otimes {\mc Cl}(1).$$ This implies that $U(\g,e)$ is independent of the choice of $\mf l$, up to isomorphism. 
The second assertion, regarding the independence of the Dynkin grading, follows analogously by using both Parts (iii) and (iv) of Theorem \ref{thm::GGrev} and Theorem \ref{thm::1}. This completes the proof.
\end{proof}

Based on the results of the preceding two sections, we conclude that for a fixed good grading $\Gamma$, there exist at most two isomorphism classes of finite $W$-superalgebras  associated to $\Gamma$, namely those of $U(\mf{g}, e)$ and $\mc W$. 
In the following section, we shall prove that these two associative superalgebras are indeed non-isomorphic in the case when ${\rm r}$ is odd.

 A Lie superalgebra appearing in Kac's list \eqref{eq::Kaclist} is called \textit{exceptional} if it is isomorphic to $D(2,1; \alpha)$, $G(3)$, or $F(4)$.
It follows from \cite[Theorem 5.1]{H12} that  all good gradings for an even nilpotent element $e$ in an exceptional Lie superalgebra are Dynkin gradings.  As a consequence of Theorem \ref{thm::6}, we obtain the following corollary:
\begin{cor} Suppose that $\g$ is exceptional. Then  the isomorphism class of the finite W-algebra $U(\g,e)$ is independent of the choice of the good grading $\Gamma$.
\end{cor}

\section{Representations of the finite $W$-superalgebra $\mc W_{\mf l}$ and Whittaker modules over $\g$} \label{sect::repW}

In the remainder of this article, unless mentioned otherwise, we fix an even nilpotent element $e\in \g_\oa$, a good grading $\Gamma$ for $e$ and a Lagrangian subspace $\mf l\subset \g(-1)$. Let $U(\g,e)$ and $\mc W({=}\mc W_{\mf l})$ denote the corresponding finite W-superalgebras.
{The goal of this section is to provide a description of irreducible representations of the finite $W$-superalgebras $\mc W$ in terms of generalized Whittaker modules over $\g$, which we shall explain in subsection~\ref{sect::Skreq}.}

\subsection{Categories of modules over superalgebras}\label{subsec:smodules}
 Throughout this paper, all modules over an associative superalgebra $A = A_\oa \oplus A_\ob$ are assumed to be $\Z_2$-graded, that is, an $A$-module $M$ decomposes as $M = M_\oa \oplus M_\ob$ satisfying $A_i M_j \subseteq M_{i+j}$ for all $i, j \in \Z_2$. 
  A homogeneous homomorphism $f: M \to N$ of $A$-modules is a linear map satisfying $f(M_i) \subseteq N_{i+\vert{}f\vert{}}$ for all $i \in \mathbb{Z}_2$, and $$f(a \cdot m) = (-1)^{\vert{}f\vert{} \vert{}a\vert{}} a \cdot f(m)$$for all homogeneous elements $a \in A$ and $m \in M$, where $f$ is called even if $\vert{}f\vert{} = \bar{0}$ and odd if $\vert{}f\vert{} = \bar{1}$.  We let $A\text{-sMod}$ denote the category of left $A$-modules whose morphisms are restricted to even homomorphisms. We have the parity change functor $\Pi: A\sMod\rightarrow  A\sMod$ with $\Pi^2$ being isomorphic to the identity functor $\text{Id}$ of $A\text{-sMod}$. It is defined on
 $M\in A$-sMod by declaring that $\Pi M$ has the same underlying vector space as $M$ but with the opposite $\Z_2$-grading, viewed as an object in $A$-sMod with the new action given by $a\cdot m:= (-1)^{|a|}am$.

\subsection{Skryabin equivalence}  \label{sect::Skreq} In this section, we allow ${\rm r}$ to be an arbitrary non-negative integer.
\subsubsection{} 
A $\mf g$-module $M$ is called a generalized Whittaker module if $a -\chi(a)$ acts on $M$ locally nilpotently,
for any $a\in \mf m$.  Let $\g\text{-Wmod}^{\chi}$ denote the full subcategory of $U(\g)\text{-sMod}$ consisting of all generalized Whittaker modules.

Let $M$ be a Whittaker module. A Whittaker vector in $M$ is a vector $v\in M$ satisfying $(a-\chi(a))v=0$, for $a\in \mf m$. We denote the subspace of Whittaker vectors in $M$ by $\text{Wh}_\chi(M)$. This defines the {\em Whittaker functor}  $$\text{Wh}_\chi(\_):  \g\text{-Wmod}^{\chi}\rightarrow U(\g,e)\text{-sMod},~~M\mapsto \text{Wh}_\chi(M).$$
We define  another functor $${\rm Sk}(\_): U(\g,e)\text{-sMod} \rightarrow \g\text{-Wmod}^{\chi} ,~~V\mapsto Q_{\mf l}\otimes_{U(\g,e)}V.$$ 

The following theorem was originally established by Skryabin   \cite{Skr02} for reductive Lie algebras, and was   generalized to the Lie superalgebra setting  in \cite{Zh14, ZS15, SX20}:

\begin{thm}[Skryabin equivalence] \label{thm::Skr}  The functors ${\rm Wh}_\chi$ and  ${\rm Sk}$ form mutually inverse equivalences between $\g\text{-}{\rm Wmod}^{\chi}$ and $U(\g,e)\sMod$.
\end{thm}

\subsubsection{} \label{sect::322}
In the case of semisimple Lie algebras, the existence of finite-dimensional representations of $U(\g,e)$ has been proved by Premet in \cite[Corollary 4.1]{Pr07} and they have been studied in various works (e.g., \cite{Lo10, Lo11}).   The primary goal of this section is to  explore the relationship between the finite-dimensional representations of the finite $W$-superalgebra associated with $\g$ and those of the finite $W$-algebras associated with its even part $\g_\oa$.

Let $U(\g_\oa, e)$ be the finite $W$-algebra associated to $e\in \g_\oa$. Here, the construction is based on the restricted good grading $\Gamma|_{\g_\oa}$ to $\g_\oa$ for $e$ and the corresponding nilpotent subalgebra $\mathfrak{m}_{\bar{0}}$. Denote by $\g_\oa\text{-Wmod}^{\chi^0}$ the category of all generalized Whittaker modules over $\g_\oa$ associated to the character $\chi^0:=\chi|_{\mf m_\oa}: \mf m_\oa\rightarrow \C$. Then we have the classical Skryabin equivalence of categories:
\[\text{Wh}^0_\chi(\_):  \g_\oa\text{-Wmod}^{\chi^0}\rightarrow U(\g_\oa,e)\text{-{sMod}},~\hskip0.1cm {\rm Sk}^0(\_): U(\g_\oa,e)\text{-{sMod}} \rightarrow \g_\oa\text{-Wmod}^{\chi^0},\] where $\text{Wh}_\chi^0$ and ${\rm Sk}^0$ are $\g_\oa$-analogues of the functors $\text{Wh}_\chi$ and ${\rm Sk}$.

Let $\Res(\_): U(\g){\sMod}\rightarrow U(\g_\oa){\sMod}$ denote the restriction functor. This functor admits a left adjoint given by the induction  functor 
$\Ind(-) = U(\g) \otimes_{U(\g_{\oa})}(\_).$  For any $M\in \g_\oa\Wmod^{\chi^0}$, we may note that $\Res\Ind(M)\in \g_\oa\Wmod^{\chi^0}$. It follows that for any $a\in \mf m_\ob$ the element  $a^2 = \frac{1}{2}[a,a]$ acts locally nilpotently on $\Ind(M)$. Therefore, $\Ind(M)\in \g\Wmod^\chi$. This allows us to define the exact and faithful functors $ \mathsf{Ind}(\_)$ and $\mathsf{Res}(\_)$ between $U(\g, e){\sMod}$ and $  
U(\g_{\bar{0}}, e)$-${\mathrm{sMod}}$ such that the following diagrams commute 
\begin{equation} \label{diagram::1}
\begin{tikzcd}
U(\g, e){\sMod}  & \g\Wmod^\chi \arrow[l, "{\rm Wh}_\chi"'] \\
U(\g_{\bar{0}}, e)\text{-}{\mathrm{sMod}} \arrow[u, "\mathsf{Ind}(-)"] \arrow[r, "{\rm Sk}^0"] & \g_{\bar{0}}\Wmod^{\chi^0} \arrow[u, "\mathrm{Ind}(-)"]
\end{tikzcd}
\begin{tikzcd}
U(\g, e){\sMod} \arrow[r, "{\rm Sk}"] \arrow[d, "\mathsf{Res}(-)"'] & \g\Wmod^\chi \arrow[d, "\mathrm{Res}(-)"'] \\
U(\g_{\bar{0}}, e)\text{-}{\mathrm{sMod}}  & \g_{\bar{0}}\Wmod^{\chi^0} \arrow[l, "{\rm Wh}^0_\chi"']
\end{tikzcd}
 \end{equation}

We denote by $U(\mf{g},e)\text{-}\mathrm{sMod}_{\mathrm{fd}}$ and $U(\mf{g}_\oa,e)\text{-}\mathrm{sMod}_{\mathrm{fd}}$ the categories of finite-dimensional $U(\g,e)$- and $U(\g_\oa,e)$-modules with even homorphisms, respectively. By \cite[Theorem 1.2.3-(1)]{Lo10}, there exists a finite-dimensional $U(\g_\oa,e)$-module, and hence the category $U(\mf{g}_\oa,e)\text{-}\mathrm{sMod}_{\mathrm{fd}}$ is nonzero. The following proposition shows that $U(\g,e)\text{-}\mathsf{sMod}_{\mathrm{fd}}$ is nonzero as well.

\begin{prop}\label{prop::FrobeniusExt}
The functors $\mathsf{Ind}$ and $\mathsf{Res}$ restrict to well-defined functors between the categories $U(\mathfrak{g}_{\bar{0}},e)\text{-}\mathrm{sMod}_{\mathrm{fd}}$ and $U(\mathfrak{g},e)\text{-}\mathrm{sMod}_{\mathrm{fd}}$. Furthermore, $\mathsf{Ind}$ is left adjoint to $\mathsf{Res}$  on these subcategories. In particular, every finite-dimensional representation of $U(\g,e)$ is obtained as a quotient of $\mathsf{Ind}(M)$, for some $M \in U(\mathfrak{g}_\oa,e)\text{-}\mathrm{sMod}_{\mathrm{fd}}$.
\end{prop}

Before proceeding to the proof of the above proposition, we remark that the following lemma is of independent interest, as it applies to a general setting beyond the current one:
\begin{lem} \label{lem::nilpotentlem}
Let $\mf a=\mf a_{\bar 0}\oplus\mf a_{\bar 1}$ be a finite-dimensional nilpotent Lie superalgebra acting on a space $V$. Suppose that $\chi:\mf a\rightarrow\C$ is a character of $\mf a$, which therefore restricts to a character of $\mf a_{\bar 0}$. Then the subspace $V^{\mf a^\chi}:=\{v\in V\mid av=\chi(a)v,\forall a\in \mf a\}$ is finite dimensional if and only if the subspace $V^{\mf a_\oa^\chi}:=\{v\in V\mid av=\chi(a)v,\forall a\in \mf a_{\bar 0}\}$ is finite dimensional.
\end{lem}
\begin{proof} It suffices to show that $\dim V^{\mf a^\chi}<\infty$ implies that $\dim V^{\mf a_\oa^\chi}<\infty$. Let $\{Y_1,\ldots, Y_d\}\subset \mf a_\ob$ be a basis for $\mf a_\ob$ such that $$[\mf a_\oa, Y_k]\subset \text{span}\{Y_1,\ldots, Y_{k-1}\},$$ for any $1\leq k\leq d$ (set $Y_0:=0$). Define a filtration of subspaces of $V^{\mf a_\oa^\chi}$ as follows:
     \[ V^{\mf a^\chi} =V_d\subset V_{d-1}\subset \cdots \subset V_1 \subset  V_0 =V^{\mf a_\oa^\chi},\]
     where $V_k:=\{v\in V^{\mf a^\chi_\oa}\mid Y_1v=Y_2v=\cdots= Y_kv=0\}$. Since $x^2v =\frac{1}{2}[x,x]v =0$, for any $x\in \mf a_\ob$ and $v\in V^{\mf a_\oa^\chi}$, it follows that $$Y_i Y_j v = -Y_jY_iv,~~~\text{and }~Y_i^2v=0,$$
     for any $1\leq i, j\leq d$ and $v\in V^{\mf a_\oa^\chi}$. This implies that $Y_k V_{k-1} \subset V_{k-1}$, for each $k$. Therefore, we have well-defined nilpotent operators 
     \[T_k: V_{k-1}\rightarrow V_{k-1}, ~v\mapsto Y_kv,~\text{ for }1\leq k\leq d.\] We may note that $V_k = \ker(T_k)$, for each $k$. Since $T_k^2=0$ we have $\text{Im}T_k\subseteq\ker T_k$ and thus $\dim V_{k}<\infty$ implies that $\dim V_{k-1}<\infty$. Indeed, in this case, we have 
     \[\dim V_{k-1} = \dim (T_k(V_{k-1}))+\dim V_k \leq 2\dim(V_k).\]
    Consequently, we obtain that  \[\dim(V^{\mf a^\chi_\oa})= \dim V_{0}  \leq 2^d\dim(V_d)=2^d \dim(V^{\mf a^\chi}).\] This completes the proof.
\end{proof}

\begin{proof}[Proof of Proposition \ref{prop::FrobeniusExt}]
   First, we show that $\mathsf{Ind}(\_): U(\mf{g}_\oa,e)\text{-}\mathrm{sMod}_{\mathrm{fd}} \rightarrow U(\mf{g},e)\text{-}\mathrm{sMod}_{\mathrm{fd}}$ is well-defined. To see this, let $M\in \g_\oa\Wmod^{\chi^0}$ be such that $\text{Wh}^0_\chi(M)$ is finite-dimensional. Denote by $E:= U(\g_\ob)$ the $\g_\oa$-module under the adjoint action of $\g_\oa$. By \cite[Theorem 8.1]{BK02}, it follows that  $\text{Wh}^0_\chi(E\otimes M)\cong E \otimes   \text{Wh}_{\chi}^0(M)$ as vector space. This implies that $\text{Wh}^0_\chi(E\otimes M)$ is finite-dimensional. Therefore, $\text{Wh}_{\chi}(\Ind(M))\subseteq \text{Wh}^0_{\chi}(E\otimes M)$ is finite-dimensional. 

   Next, the well-definedness of $\mathsf{Res}$ on the categories of finite-dimensional $U(\mf{g},e)$- and $U(\mf{g}_\oa,e)$-modules follows from   Lemma \ref{lem::nilpotentlem} by setting $\mathfrak{a} = \mathfrak{m}$. Since the  functors  $\Ind$ and  $\Res$  form a pair of adjoint functors between $\g_\oa\text{-Wmod}^\chi$ and $\g\text{-Wmod}^\chi$, the same holds for the functors $(\mathsf{Ind}, \mathsf{Res})$ by the commutativity of the diagram \eqref{diagram::1}. Finally, any $U(\g,e)$-module $X$ can be realized as a quotient of $\mathsf{Ind}\mathsf{Res}(X)$. If, in addition, $X$ is finite-dimensional, then so is $\mathsf{Res}(X)$. This completes the proof. 
\end{proof}

\subsection{Irreducible representations of  $\mc W$} \label{sect::IrrWell}
 In the case when ${\rm r}$ is even, then $\mc W = U(\g,e)$ and thus the representation theory reduces to the study of the category of  Whittaker modules over $\g$ by Theorem~\ref{thm::Skr}. Therefore, we shall focus on the case when $\rm r$ is odd, and make this assumption for the remainder of this paper.

 \subsubsection{Clifford Twist} \label{sect::Cltw}
 Following \cite{KKT16}, we recall the basic framework of supercategories.  A \emph{supercategory} is a triple $(\mc C, P, \xi)$ consisting of a category $\mc C$ and an endofunctor $P \colon \mc C \to \mc C$ equipped with a natural isomorphism $\xi \colon P^2 \xrightarrow{\cong} \mathrm{id}_{\mc C}$.  For instance, for any associative superalgebra $A$, the category $\mc C = A\text{-}\mathrm{sMod}$ of all $A$-modules with parity-preserving homomorphisms, equipped with the parity-change functor $P= \Pi$ and the natural isomorphism $\xi^A: \Pi^2 \xrightarrow{\cong}\mathrm{id}_{\mc C}$, is a supercategory. In this setup, one can define notions of \emph{superfunctors} and \emph{equivalences of supercategories}; see also \cite[Definition~2.1]{KKT16}. A \emph{Clifford twist} of a supercategory $(\mc C, P, \xi)$ is the supercategory $(\mc C^{\texttt{CT}}, P^{\texttt{CT}}, \xi^{\texttt{CT}})$ consisting of objects  of pairs $(X, \varphi)$, where $X \in \mc C$ and $\varphi \colon P X \xrightarrow{\cong} X$ is an isomorphism in $\mc C$ making the following diagram commute:
	\[
	\begin{tikzcd}
		P^2 X \arrow[r, "\xi_X"] \arrow[d, "P \varphi"'] & X \\
		P X \arrow[ur, "\varphi"'] &
	\end{tikzcd}
	\]
	For two objects $(X, \varphi)$ and $(Y, \psi)$ in $\mc C^{\texttt{CT}}$, the morphism space $\mathrm{Hom}_{\mc C^{\texttt{CT}}}\big((X, \varphi), (Y, \psi)\big)$ is defined as the subspace of $\mathrm{Hom}_{\mc C}(X, Y)$ consisting of morphisms $f \colon X \to Y$ such that the diagram
	\[
	\begin{tikzcd}
		P X \arrow[r, "P f"] \arrow[d, "\varphi"'] & P Y \arrow[d, "\psi"] \\
		X \arrow[r, "f"'] & Y
	\end{tikzcd}
	\] 	commutes. The associated functors $P^{\texttt{CT}}$ and $\xi^{\texttt{CT}}$ are given by
	\begin{align*}
	&P^{\texttt{CT}}(X, \varphi) = (X, -\varphi),~ \quad \xi^{\texttt{CT}}_{(X, \varphi)}=\mathrm{id}_{(X, \varphi)}.
	\end{align*} When the functor $P$ and isomorphism $\xi$ are clear from the context, we simply write $\mc C$ for the supercategory $(\mc C, P, \xi)$. Recall that for a full supercategory $\mc C$ of $A\text{-sMod}$, we  denote by   $\mathrm{Irr}^{\texttt{M}}(\mc C)$ and $\mathrm{Irr}^{\texttt{Q}}(\mc C)$  the sets of isomorphism classes of simple objects of type~$\texttt{M}$ and simple objects of type~$\texttt{Q}$ in $\mc C$, respectively.

\subsubsection{Categorical Equivalence}
It is well known that a simple module $M$ over an associative superalgebra $A$ admits either no odd automorphisms or a unique (up to scalar) odd  automorphism; see, e.g., \cite[Section 3.1]{CW12} for details. Accordingly, $M$ is said to be of {\it type} \texttt{Q} if $M\cong \Pi M$ in $A\sMod$,   and of {\it type} \texttt{M} otherwise. 

Recall that $G \colon U(\g,e)\text{-sMod} \to \mc W\text{-sMod}$ and $F \colon \mc W\text{-sMod} \to U(\mathfrak{g},e)\text{-sMod}$ denote the restriction and induction functors, respectively. Recall further that $\ov \theta$ denotes the image of $\theta$ under the canonical projection from $U(\g)$ onto $U(\g,e)$. By Theorem \ref{thm::1}, we can identify ${\mc Cl}(1)$ with a subalgebra of $U(\g,e)$ via the mapping $\Theta \mapsto \ov \theta$ and the multiplication defines a superalgebra isomorphism
\begin{align*}
&\mc W\otimes {\mc Cl}(1) \xrightarrow{\cong} U(\g,e),~~~ f\otimes g\mapsto fg. 
\end{align*}  In particular, $U(\mathfrak{g},e)$ is free as a $\mc W$-module, and hence both $F$ and $G$ are exact functors. These functors were  also introduced in \cite{C95, BK02} in slightly different formulations for the study of symmetric group representations, and they constitute a Morita super-equivalence in the terminology of \cite{W09}.  

By \cite[Proposition~8.4]{C95} (see also \cite[Section 3.1]{CW12}), $F$ sends a simple module of type \texttt{M} to one of type \texttt{Q}, and a simple module of type \texttt{Q} to a direct sum of two simple modules of type \texttt{M} of opposite parity.

The following lemma, which may be of independent interest, is a special case of \cite[Lemma 2.2]{KKT16} and is similar to \cite[Theorem 3.4]{BK02}. As the formulations there are slightly different, we provide a self-contained proof below for the reader's convenience.
\begin{lem} \label{lem::bij}
For any $X \in U(\mf{g},e)\textnormal{\text{-sMod}}$ and $Y \in \mc W\textnormal{\text{-sMod}}$, there are the following isomorphisms:
\begin{equation} \label{eq::isoFT}
    F(G(X)) \cong X \oplus \Pi X \quad \text{and} \quad G(F(Y)) \cong Y \oplus \Pi Y.
\end{equation}  
\end{lem}
\begin{proof} The second isomorphism in \eqref{eq::isoFT} follows by the isomorphism $U(\g,e)\cong \mc W\otimes {\mc Cl}(1)$ immediately. Thus, it remains to show the first isomorphism  in \eqref{eq::isoFT}.   Let $X\in U(\g,e)\sMod$. Then the counit $\epsilon_X$ of the adjoint pair $(F,G)$ is given by
$$\epsilon_X \colon U(\g,e) \otimes_{\mc W} G(X) \longrightarrow X, \quad r \otimes x \longmapsto rx,$$ and the kernel $\ker(\epsilon_X) = \spann\{1\otimes \ov \theta x- \ov \theta\otimes x\mid~x\in X\}$ is isomorphic to $\Pi X$ as a $\mc W$-submodule. Therefore, we have a short exact sequence in $U(\g,e)\sMod$:
\begin{align} \label{eq::splitsex}
&0 \longrightarrow \Pi X \xrightarrow{\quad j \quad} U(\g,e) \otimes_{\mc W} G(X) \xrightarrow{\quad \epsilon_X \quad} X \longrightarrow 0, 
\end{align} where $j(x) =1\otimes \ov  \theta x- \ov \theta\otimes x$ for any $x\in G(X)$. We claim that the short exact sequence in \eqref{eq::splitsex} splits. To see this, we define a homomorphism of $U(\g,e)$-modules $i_X: X\rightarrow F(G(X))$ by $$i_X(x) = \frac{1}{2}(1 \otimes x + \ov \theta \otimes \ov \theta x),~x\in G(X).$$ Then $\epsilon_X\circ i_X$ is the identity map on $X$. This completes the proof.  
\end{proof}


 The following useful lemma is   analogous to \cite[Lemma 2.7]{KKT16} with the same proof.
\begin{lem} \label{lem::KKT27} For any associative superalgebra $A$, there exists an equivalence of supercategories
\[
\mathbb{E}^{A} \colon (A\otimes {\mc Cl}(1)\textnormal{\text{-sMod}}, \Pi) \xrightarrow{\sim} (A\textnormal{\text{-sMod}}^{\mathtt{CT}}, \Pi^{\mathtt{CT}}),
\]
sending $M$ to $(M, \phi_M)$, where $\phi_M(m) = \Theta m$ for any $m\in M$.
\end{lem}

We are now in a position to prove the following theorem, which is a restatement of Theorem C in Section \ref{sect::introThmC}.
\begin{thm} \label{thm::rethmC}
There is an equivalence of supercategories  $$\mathbb E \colon (U(\g,e)\emph{-sMod}, \Pi) \to (\mc W\emph{-sMod}^{\textnormal{\texttt{CT}}},\Pi^{\textnormal{\texttt{CT}}})$$ such that 
        	\begin{align*}
        	&\mathbb E(M) = (G(M),\varphi_M),~\text{where }\varphi_M(m)= {\ov \theta} m \text{ for }m\in M.
        	\end{align*}          Furthermore,  $F$ and $G$ are superfunctors between $(\mc W
            \emph{-sMod}, \Pi)$  and  $(U(\g,e)\emph{-sMod}, \Pi)$, and they induce two-to-one maps \begin{align}&\emph{Irr}^{\textnormal{\texttt{M}}}(\mc W\emph{-sMod})\rightarrow \emph{Irr}^{\textnormal{\texttt{Q}}}(U(\g,e)\emph{-sMod}),\label{eq::1411} \\ &\emph{Irr}^{\textnormal{\texttt{M}}}(U(\g,e)\emph{-sMod})\rightarrow \emph{Irr}^{\textnormal{\texttt{Q}}}(\mc W\emph{-sMod}), \label{eq::1412}
        	\end{align} respectively.  
\end{thm}
\begin{proof}
By Theorem \ref{thm::1}, we have an isomorphism $U(\g,e)\cong \mc W\otimes {\mc Cl}(1)$. This, together with Lemma \ref{lem::KKT27}, yields the desired equivalence $\mathbb E$ of supercategories.
 Furthermore, the statements concerning the two-to-one maps in \eqref{eq::1411} and \eqref{eq::1412}
 follow directly from \cite[Lemma 2.2]{KKT16}. Alternatively, they can also be deduced from Lemma \ref{lem::bij}. This completes the proof.
\end{proof}

 To each simple $U(\g,e)$-module $N$, we fix a simple $\mc W$-submodule $G'(N)\subseteq G(N)$. Recall that  $\widetilde{G}(M) = G'(\operatorname{Wh}_\chi(M))$, for any  irreducible generalized Whittaker module $M\in \g\Wmod^\chi$. By Lemma \ref{lem::bij} and Theorem~\ref{thm::rethmC}, $G(\operatorname{Wh}_{\chi}(M))$ is a simple module of type $\texttt{Q}$ if $M$ is of type $\texttt{M}$, and is isomorphic   to a direct sum of two type $\texttt{M}$ simple modules $\widetilde{G}(M)\oplus \Pi\widetilde{G}(M)$ otherwise.

The following corollary is a strengthened version of Corollary \ref{cor::1st} in Subsection \ref{sect::DR}.

\begin{cor}
\label{thm::UWsimples}  
The assignment
\[
\operatorname{Irr}^{\mathtt{M}}(\g\Wmod^\chi) \to \operatorname{Irr}^{\mathtt{Q}}({\mc W}\emph{-sMod}), \quad [M] \mapsto [G(M)]
\]
gives rise to a two-to-one map, whereas the assignment
\[
\operatorname{Irr}^{\mathtt{Q}}(\g\Wmod^\chi) \to \operatorname{Irr}^{\mathtt{M}}({\mc W}\emph{-sMod}), \quad [M] \mapsto [\widetilde{G}(M)]
\]
is an injective map such that every simple module in $\operatorname{Irr}^{\mathtt{M}}({\mc W}\emph{-sMod})$ is isomorphic to $\widetilde{G}(M)$ or $\Pi\widetilde{G}(M)$ for some $[M]$. 
\end{cor}
\begin{proof}
This follows immediately from Theorems \ref{thm::Skr} and \ref{thm::rethmC}.
\end{proof}

\subsubsection{Finite-dimensional representations of ${\mc W}$} 

This section is devoted to the study of the relationship between finite-dimensional $U(\g,e)$-modules and those of ${\mc W}$, and to show that the finite $W$-superalgebras $U(\g,e)$ and ${\mc W}$ are not isomorphic. 
\begin{prop}\label{res::fdim} We have
\begin{itemize}
    \item[(i)] The algebra
${\mc W}$ admits finite-dimensional representations.
    \item[(ii)] The sets
\begin{equation*}
\begin{aligned}
& \{ G(M) \mid M \in \operatorname{Irr}^{\mathtt{M}}(U(\g,e)\textnormal{\text{-sMod}}_{\mathrm{fd}}) \}, \quad \text{and} \\
& \{ G'(M),\, \Pi G'(M) \mid M \in \operatorname{Irr}^{\mathtt{Q}}(U(\g,e)\textnormal{\text{-sMod}}_{\mathrm{fd}}) \}
\end{aligned}
\end{equation*}
provide exhaustive lists of all finite-dimensional irreducible ${\mc W}$-modules of type $\mathtt{Q}$ and type $\mathtt{M}$, respectively. 
    \item[(iii)]  For any simple $U(\g,e)$-module $X$ of  type $\texttt{M}$  and any simple $U(\g,e)$-module $Y$ of  type~$\texttt{Q}$, the following equalities hold:
    \begin{align} \label{eq::dimSimples}
&\dim G(X) = \dim X \text{ and } \dim G'(Y) =\frac{1}{2}\dim Y,
\end{align}
\end{itemize}
\end{prop}
\begin{proof}
Since the proof of Part (ii) is already contained in the proof of Corollary \ref{thm::UWsimples}, it remains to prove Parts (i) and (iii). As we have alluded to earlier, there is a finite-dimensional $U(\g,e)$-module by applying Proposition \ref{prop::FrobeniusExt} and \cite[Theorem 1.2.3-(1)]{Lo10}, thereby completing the proof of Part (i). Finally, the proof of Part (iii) follows by Lemma \ref{lem::bij} and Theorem \ref{thm::rethmC}. This completes the proof.   \end{proof}

\begin{rem} \label{res::fdim2}
It was conjectured in \cite[Conjecture 1.3 (2)]{ZS17} that the finite $W$-superalgebra $U(\g, e)$ always admits a two-dimensional irreducible representation. A sufficient condition for this  conjecture was subsequently established in \cite[Proposition~4.3]{ZS19},  which presumes the validity of the conjecture  \cite[Conjecture 4.2]{ZS19} that ${\mc W}$ always admits a one-dimensional representation. Here, we can provide an alternative approach that not only recovers this sufficiency but also gives a more refined characterization of the interplay between these two conjectures. Namely, the following equivalences are established via \eqref{eq::dimSimples}: 
\begin{itemize}
    \item[$\bullet$]   ${\mc W}$ affords a one-dimensional representation if and only if $U(\g, e)$ affords a two-dimensional irreducible representation of type $\texttt{Q}$. 
    \item[$\bullet$]   ${\mc W}$ affords a  two-dimensional irreducible representation of type $\texttt{Q}$  if and only if   $U(\g, e)$ affords a two-dimensional irreducible representation of type ${\texttt 
    M}$. 
\end{itemize}

If $e$ is principal in $\g_\oa$, then all irreducible $U(\g, e)$-modules are finite-dimensional; see, e.g., \cite[Proposition 3.7]{PS16}. It follows that all irreducible ${\mc W}$-modules are likewise finite-dimensional.   
\end{rem}

We conclude this subsection with the following corollary, which states that for a fixed good grading $\Gamma$, there exist exactly two isomorphism classes of finite $W$-superalgebras associated to $\Gamma$.
\begin{cor} \label{cor::NonIso}
Let $\Gamma$ be a  good grading for $e$. Then the  finite $W$-superalgebras ${\mc W}$ and $U(\g,e)$ associated to $\Gamma$ are not isomorphic. 
\end{cor}
\begin{proof}
Let $\mc F_1:=U(\mf{g},e)\text{-}\mathrm{sMod}_{\mathrm{fd}}$ and $\mc F_2:={\mc W}\text{-}\mathrm{sMod}_{\mathrm{fd}}$ be the categories of  finite-dimensional $U(\g,e)$- and $\mc W$-modules, respectively. By Proposition \ref{res::fdim}, both categories are non-empty.  Suppose on the contrary that $\mc W\cong U(\g,e)$. Then we have a dimension- and type-preserving isomorphism $f:\mc F_1 \rightarrow \mc F_2$. Let $N\in\mc F_1$ be an object of minimal dimension, and set  $d:=\dim N$. If $N$ is of type $\texttt{Q}$, it follows from Parts (ii), (iii) of  Proposition \ref{res::fdim} that 
\begin{align*}
&\min\{ \dim G'(M) \mid M \in \operatorname{Irr}^{\mathtt{Q}}(\mc F_1)\} = \frac{1}{2}d < d,
\end{align*} a contradiction. Therefore,  $N$ is of type $\texttt{M}$ and so is $f(N)$. By Parts (ii), (iii) of Proposition~\ref{res::fdim} again, we conclude that $G(f(N)) \in \operatorname{Irr}^{\mathtt{Q}}(\mathcal{F}_1)$ reaches the minimal dimension $d$. Applying the same reasoning to $G(f(N))$ again leads to a contradiction. This completes the proof.\end{proof}

\section{Applications and Examples} \label{sect::eg}

\subsection{Finite-dimensional  $U(\g,e)$-modules} \label{sect::appendix} 
     Retaining the notation and setup from Section~\ref{sect::repW}, this subsection is devoted to the study of finite-dimensional modules over the finite W-superalgebra $U(\g,e)$ associated to a basic classical  Lie superalgebra $\g$ with an even nilpotent  element $e$, as constructed in Subsection~\ref{sect::PrDef}. We set $U(\g_\oa, e)$ to be the finite $W$-algebra associated to $\g_\oa$ corresponding to the nilpotent element $e\in \g_\oa$ with the restricted good grading $\Gamma|_{\g_\oa}$ to $\g_\oa$ for $e$ and the corresponding nilpotent subalgebra $\mathfrak{m}_{\bar{0}}$.
	
	  First, we briefly review the approach to prove existence of finite-dimensional representations of finite $W$-algebras $U(\g_\oa,e)$  in \cite{Pr07, Pr07b, Lo10, Lo11}.  Let $G$ be the simply connected reductive Lie group with Lie algebra $\g_\oa$. Let ${\mathbb O}:= G\cdot \chi$ denote the coadjoint orbit of $\chi$, and $\ov{\mathbb O} = \ov{G\cdot \chi}$ its closure. For a primitive ideal $I$ of $U(\g_\oa)$, we let $\mc V(I)\subset \g_\oa^\ast$ denote its associated variety; see, e.g., \cite[Subsection 3.2]{Pr07b}.  Premet \cite[Theorem 3.1]{Pr07b} proved that if  $I=\Ann_{U(\g_\oa)}({\rm Sk}^0(N))$ is the annihilator ideal of ${\rm Sk}^0(N)$ for some   $U(\g_\oa,e)$-module $N$, then 
		\begin{align}
			&\mc V(I) \supseteq \ov{\mathbb O}, \label{eq::Prm}
		\end{align} with equality holding when $N$ is finite-dimensional. 
		
			Denote by $\mc I_{\mathbb O}$ the set of all primitive ideals $I$ of $U(\g_\oa)$ such that $\mc V(I) = \ov{\mathbb O}$.  To further explore the finite-dimensional representations  of $U(\g_\oa,e)$, Premet in \cite[Conjecture~3.2]{Pr07b} conjectured  that any ideal 
			 $I\in \mc I_{\mathbb O}$ coincides with $\mathrm{Ann}_{U(\g_\oa)}{{\rm Sk}^0}(N)$ for some finite-dimensional irreducible $U(\g_\oa,e)$-module $N$. This conjecture was subsequently established by Premet \cite[Theorem 1.1]{Pr07} under certain conditions on $I$, and was later settled in full generality by Losev \cite[Theorem~1.2.2]{Lo10}. This implies that an irreducible $U(\g_\oa,e)$-module $N$ is finite-dimensional if and only if $ \Ann_{U(\g_\oa)}{\rm Sk}^0(N)\in \mc I_{\mathbb O}$; see also \cite[Subsection 5.2]{Lo10b}.  As a consequence, $U(\g_\oa,e)$ admits finite-dimensional irreducible representations; see \cite[Corollary 4.1]{Pr07}.

	Recall from Subsection \ref{sect::322} the adjoint pair $(\mathsf{Ind}, \mathsf{Res})$ of exact functors between the categories of $U(\g_\oa,e)$- and $U(\g,e)$-modules. By Lemma~\ref{lem::nilpotentlem},  a $U(\g,e)$-module $N$ is finite dimensional if and only if the $U(\g_\oa, e)$-module $\mathsf{Res}(N)$ is finite dimensional. This allows us to determine  finite-dimensionality of $U(\g,e)$-modules in a manner analogous to the one described above. Indeed, from Losev's characterization, we have that a finitely generated $U(\g_\oa,e)$-module $N$ is finite dimensional if and only if $\mc V\left(\mathrm{Ann}_{U(\g_\oa)}\left({\rm Sk}^0(N)\right)\right)=\ov{\mathbb O}$.
    Properties of our restriction functor now imply the following.
    \begin{cor} Let $N$ be a finitely generated $U(\g,e)$-module. Then $N$ is finite dimensional if and only if
        $\mc V\left(\mathrm{Ann}_{U(\g_\oa)}\left(\Res{\rm Sk}(N)\right)\right)=\ov{\mathbb O}$.
    \end{cor}

    We can further characterize the finite-dimensionality of representations of $U(\mathfrak{g},e)$ in terms of primitive ideals in $\mathcal{I}_{\mathbb{O}}$ as follows:

	   \begin{prop} Let $N$ be a   $U(\g,e)$-module. Then the following are equivalent:
	   	\begin{itemize}
	   		\item[(i)] $N$ is finite dimensional.
	   		\item[(ii)] $\mathsf{Res}(N)$ is of finite length,  and all the annihilator ideals of the composition factors of $\Res{\rm Sk}(N)$ lie in $\mc I_{\mathbb O}$. 
	   		\item[(iii)] $\mathsf{Res}(N)$ is of finite length, and  $${\emph \Ann}_{U(\g)}({\rm Sk}(N))\supseteq {\emph \Ann}_{U(\g)}(U(\g)/U(\g)I),$$ for some $I \in \mc I_{\mathbb O}$.
	   	\end{itemize}
	   \end{prop}
 We note that the ideal $\mathrm{Ann}_{U(\g)}(U(\mathfrak{g})/U(\g)I)$ in (iii) coincides with the largest two-sided ideal of $U(\g)$ contained in $U(\g)I$.
	   \begin{proof}		 
        First, we prove that (i) and (ii) are equivalent. Suppose that Condition (i) holds, namely, $N$ is finite-dimensional.	By Lemma~\ref{lem::nilpotentlem}, $\mathsf{Res}(N)$ is finite-dimensional. This implies that   $\Res({\rm Sk}(N))$ is of finite length, and it follows from  \cite[Theorem~3.1]{Pr07b} that  any annihilator ideal of a composition factor of $\Res{\rm Sk}(N)$ lies in $\mc I_{\mathbb O}$. This  proves that (i) implies (ii). Next, suppose that Condition (ii) holds.  Let $L$ be a composition factor of $\Res {\rm Sk}(N)$.  Then it follows that $\mc V(\Ann_{U(\g_\oa)}(L)) \subseteq \mc V(\Ann_{U(\g_\oa)}(\Res {\rm Sk}(N)))$ by definition. Furthermore, we have  $$\ov{\mathbb O}\subseteq \mc V(\Ann_{U(\g_\oa)}(L)),\text{ and }\mc V(\Ann_{U(\g_\oa)}(\Res {\rm Sk}(N))) = \ov{\mathbb O}$$ by \cite[Theorem~3.1]{Pr07b}. Therefore, $\ov{\mathbb O}= \mc V(\Ann_{U(\g_\oa)}(L))$ and so $\text{Wh}_{\chi}^0L$ is finite-dimensional by \cite[Theorem 1.2.2]{Lo10}. Therefore, $\text{Wh}_{\chi}^0 (\Res {\rm Sk}(N))$ is finite-dimensional, since it is of finite length. By Lemma \ref{lem::nilpotentlem}, this implies that $N= \text{Wh}_\chi({\rm Sk}(N))$ is finite-dimensional.

	   	Next, we shall show the equivalence of (i) and (iii). Suppose that $N$ is finite-dimensional. By Proposition \ref{prop::FrobeniusExt}, there is a finite-dimensional simple $U(\g_\oa,e)$-module $X$ such that $N$ is a quotient of $\mathsf{Ind}X$ so that   $I:=\Ann_{U(\g_\oa)}({\rm Sk}^0X)\in\mc I_{\mathbb O}$. Furthermore, \cite[Lemma ~4.4]{CoM16} says that $\Ann_{U(\g)}(\Ind{\rm Sk}^0(X)) = \Ann_{U(\g)}(U(\g)/U(\g)I)$. Thus, \[\Ann_{U(\g)}({\rm Sk}(N))\supseteq  \Ann_{U(\g)}(\Ind{\rm Sk}^0(X)) =\Ann_{U(\g)}(U(\g)/U(\g)I). \] This proves that (i) implies (iii).
	   	
	   	Finally, suppose that the conditions in (iii) hold, that is, $\Res(N)$ is of finite length and  $$\Ann_{U(\g)}({\rm Sk}(N))\supseteq \Ann_{U(\g)}(U(\g)/U(\g)I),$$ for some $I \in \mc I_{\mathbb O}$. 
	   	By \cite[Theorem 1.2.2]{Lo10}, there is a finite-dimensional simple $U(\g_\oa,e)$-module $X$ such that $\Ann_{U(\g_\oa)}({\rm Sk}^0(X)) = I$. 
        We have $\Ann_{U(\g)}({\rm Sk}(N))\supseteq \Ann_{U(\g)}(\Ind {\rm Sk}^0(X))$, and so 
	   	\begin{align*}
	   	&\Ann_{U(\g_\oa)}(\Res {\rm Sk}(N)) = \Ann_{U(\g)}({\rm Sk}(N)) \cap U(\g_\oa) \\ &\supseteq \Ann_{U(\g)}(\Ind {\rm Sk}^0(X))  \cap U(\g_\oa)  =  \Ann_{U(\g_\oa)}(\Res \Ind {\rm Sk}^0(X)) \\
        & =  \Ann_{U(\g_\oa)}(E\otimes  {\rm Sk}^0 (X)),
	   	\end{align*}  where $E = U(\g_\ob)$ under the adjoint action of $\g_\oa$. As seen in the proof of Proposition~\ref{prop::FrobeniusExt}, the fact that $\dim \left(\text{Wh}_{\chi}^0 ({\rm Sk}^0 (X))\right) <\infty$ implies that $\dim \left(\text{Wh}_\chi^0(E\otimes {\rm Sk}^0 (X))\right)<\infty$ by \cite[Theorem~8.1]{BK08}. Consequently, the associated variety $\mc V(\Ann_{U(\g_\oa)}E\otimes {\rm Sk}^0 X)$ coincides with $\ov{\mathbb O}$ by \cite[Theorem 3.1]{Pr07b}, and so we have 
	   	\[\ov{\mathbb O}\subseteq \mc V(\Ann_{U(\g_\oa)}(\Res {\rm Sk}(N))) \subseteq  \mc V(\Ann_{U(\g_\oa)}(E\otimes {\rm Sk}^0 X)) =\ov{\mathbb O}.\] Consequently, we obtain  $\ov{\mathbb O}= \mc V(\Ann_{U(\g_\oa)}(\Res {\rm Sk}(N))).$ This implies that any annihilator ideal of a subquotient of $\Res {\rm Sk}(N)$ lies in $\mc I_{\mathbb O}$, and so $\text{Wh}_\chi^0(\Res {\rm Sk}(N))$ is finite-dimensional by \cite[Theorem 1.2.2]{Lo10}, as $\Res {\rm Sk}(N)$ is of finite length. Therefore, by Lemma \ref{lem::nilpotentlem}, we may conclude that $\text{Wh}_\chi ({\rm Sk}(N))$ is finite dimensional as well. This completes the proof. 
	   \end{proof}

\begin{rem}
    The results obtained in this subsection remain valid for the finite W-algebras of the queer Lie superalgebras $\mf{q}(n)$ constructed in \cite{Zh14}.
\end{rem}

\subsection{Casimir's ghost element} \label{sect::341}
	For a basic classical Lie superalgebra $\mf a$, we recall in this subsection the notion of the ghost center $\widetilde{Z} (\mf a)$ that we need for our study of finite $W$-superalgebras of $\mf{osp}(1|2n)$ in Subsection \ref{sec:exam:osp}.
    
The ghost center  $\widetilde{Z}(\mf a)= Z(\mf a)\oplus \mc A(\mf a)$ was introduced in \cite{Go01}, where $\mc A(\mf a)$ denotes the anticenter spanned by all homogeneous elements $a\in U(\mf a)$ such that $ua =(-1)^{|u|(|a|+\ov 1)}au$, for any homogeneous $u\in \mf a$. 
	Recall the Casimir's ghost element $T\in U(\mf a)_\oa$ introduced in \cite{ABF97} (see also \cite{Mu97,GL00, Go01}) is such that $T^2 \in Z(\mf a)$ and $\mc A(\mf a) =Z(\mf a)T$, and that a central character $\eta: Z(\mf a)\rightarrow \C$ is said to be {\em strongly typical} if $\eta(T^2)\neq 0$ (see \cite{Go02}).

Let $\pi: U(\g) \to Q_{\mf l}$ be the canonical projection $r \mapsto r + I_{\mf l}$. It is well-known that the restriction of $\pi$ to the center $Z(\g)$ yields an injective homomorphism into the finite $W$-superalgebra $\mc W$; see, e.g., \cite[Subsection 4.5]{ZS19}. In the following, we extend this embedding to the ghost center $\widetilde{Z}(\g)$ as follows:
\begin{align*}
&q: \widetilde{Z}(\g) = Z(\g)\oplus \mc{A}(\g)\hookrightarrow \mc W,~~(z,a) \mapsto \pi(z)+\pi(a\theta),~~\text{ for }z\in Z(\g)~\text{ and }~a\in \mc A(\g),
\end{align*}  see also \cite[Lemma 6.2, Theorem 6.5]{G24}, where the specific case of $\mf{osp}(1|2n)$ was considered.
 To see this, we may note that $\pi(T)\in U(\g,e)$, which  is a non-zero element since $T^2\neq 0$ and $\pi(T)^2$ lie in $\pi(Z(\g))$. Therefore, by Theorem \ref{thm::1} we have $$\pi(T\theta) = \frac{1}{2}[\pi(T),\pi(\theta)]\in [U(\g,e),\pi(\theta)]\subseteq \mc W.$$
    This implies that $q$ is a well-defined homomorphism of (ungraded) associative algebras. Suppose that $z_1+z_2T$ lies in the kernel of $q$, for some $z_1, z_2\in Z(\g)$, then we have $\pi(z_2T\theta)=-\pi(z_1)$ in $\mc W$. Since $\pi(T\theta)$ is odd, this implies that $\pi(z_1)=0$ and $\pi(z_2T\theta)=0$. Therefore,  $\pi(z_2)^2\pi(T\theta)^2= \pi(z_2T\theta)^2=0$ in $\pi(Z(\g))$. Since $Z(\g)$ is an integral domain, we have $\pi(z_2)=0$. This shows that $q$ is an injection.

    For any central character $\eta: Z(\mathfrak{g}) \rightarrow \mathbb{C}$, we may note that the quotient algebra $q(\widetilde{Z}(\g))/(\ker\eta)$ is generated by   an element $$q(T)~ \text{ such that } ~q(T)^2= -\eta(T^2).$$ 
Therefore,  the irreducible $q(\widetilde{Z}(\g))$-modules  are parameterized by central characters $\eta: Z(\mathfrak{g}) \rightarrow \mathbb{C}$ and are denoted by $E_{\eta}$. Here, $E_\eta$ is an irreducible module with central character $\eta$, whose structure is determined as follows: \begin{itemize}
\item If $\eta$ is not strongly typical, then $E_\eta = \mathbb{C}w$ is a one-dimensional module such that  $q(T)w=0$.
\item If $\eta$ is strongly typical, then $E_\eta$ is a two-dimensional module of type $\texttt{Q}$, spanned by $\{w, q(T)w\}$ such that $q(T)^2 w = -\eta(T^2)w$.
\end{itemize}

\subsection{Principal finite $W$-superalgebras of $\mf{osp}(1|2n)$}\label{sec:exam:osp}
As a concrete example, we focus on the ortho-symplectic Lie superalgebras  $\g:=\mf{osp}(1|2n)$, for $n\geq 1$. We  recall its standard matrix realization in the ordered basis $\{v_0,v_{-n},\ldots,v_{-1},v_1,\ldots,v_n\}$ of $\C^{1|2n}$:
$$
\mathfrak{g} = \mathfrak{osp}(1|2n) = \left\{\left( 
\begin{array}{c|cc}
0 &  y^t & -x^t \\
\hline
x & a & b \\
y & c & -a^t
\end{array} \right) \in \mathfrak{gl}(1|2n) \;\middle|\;
\begin{matrix}
a, b, c \in \C^{n\times n}, \\
x, y \in \C^{n\times 1}, \\
b = b^t, c = c^t
\end{matrix}
\right\},
$$
where $A^t$ denotes the transpose of  a matrix $A$. Let $\{e_{i,j}\}_{-n\leq i,j \leq n}$ be the standard basis of $\mathfrak{gl}(1|2n)$  and $h_i = e_{i,i} - e_{-i,-i}$, for $1 \leq i\leq n$. Let $\h$ be the Cartan subalgebra of $\mf{osp}(1|2n)$ spanned by $h_i$, for $1\leq i\leq n$. Let  $\varepsilon_i \in \mathfrak{h}^*$ be the dual basis for $\{h_i\}_{1\leq i\leq n}$ determined by $\varepsilon_i(h_j) = \delta_{i,j}$. Then we define the following root system of $\g$: 
\begin{itemize}
\item The set of positive roots $\Delta^+ = \{\varepsilon_i, 2\varepsilon_i\}_{i=1}^n \cup \{\varepsilon_i \pm \varepsilon_j\}_{1 \le i < j \le n}$.
\item The simple roots are $\Pi = \{\alpha_i\}_{1\leq i\leq n}$, where $\alpha_i = \varepsilon_i - \varepsilon_{i+1}$ for $i < n$, and $\alpha_n = \varepsilon_n$. 
\end{itemize} We note that the odd roots in $\Delta^+$ are precisely the elements $\{\varepsilon_1, \dots, \varepsilon_n\}$, all of which are non-isotropic. 
Let $\Delta^- = -\Delta^+$ and define a non-degenerate $\g$-invariant supersymmetric bilinear form on $\g$ via $(u|v) = -\text{str}(uv)$, where $\text{str}$ denotes the supertrace function. This induces a non-degenerate bilinear form $(\_|\_)$ on $\mathfrak{h}^*$ determined by $(\varepsilon_i|\varepsilon_j) =\frac{1}{2}\delta_{ij}$.  
The inner products of the simple roots under the normalized form $(\_ | \_)$ are given by:\begin{align*}(\alpha_i | \alpha_i) = 1, \quad (\alpha_i | \alpha_{i+1}) = -\frac{1}{2}, \quad (\alpha_n | \alpha_n) = \frac{1}{2},\end{align*}for $i = 1, \dots, n-1$. For each root $\alpha\in \Delta:= \Delta^+\cup \Delta^-$, let $\g^\alpha \subset \g$ denote the corresponding root space. We define a principal   nilpotent element  $e \in \g_\oa$  by $$e = \sum_{i=1}^{n-1} u_{-\alpha_i} + u_{-2\alpha_n},$$ where $u_{\alpha} \in {\mf g}^\alpha$ is a fixed non-zero root vector for each $\alpha \in \Delta$. To complete this to an $\mf{sl}(2)$-triple $\{e, h, f\} \subset \g$, we set $$\begin{aligned}
h &= \sum_{j=1}^n -(2n - 2j + 1) h_j, \\
f &= \sum_{i=1}^{n-1} i(2n-i) u_{\alpha_i} + n^2 u_{2\alpha_n}. 
\end{aligned}$$  
Then $\langle e,h,f\rangle$ forms an $\mf{sl}(2)$-triple inside $\g$. The corresponding Dynkin grading $$\Gamma: \g = \bigoplus_{-2n\leq i\leq 2n} \g(i),$$ is   determined by  $\g(0) = \mf h,~~\g(-1) = \g^{\alpha_n}$ and
\[\g(-2) = \g^{\alpha_1}\oplus \cdots \oplus \g^{\alpha_{n-1}} \oplus \g^{2\alpha_n}. \] In this specific setup,   $\mathfrak{l} = 0$ is the only isotropic subspace of $\g(-1)$, and hence we have $\mc W = \mc W_{0}$.    By \cite[Theorem 6.5]{G24}, the map $q \colon \widetilde{Z}(\g) = Z(\g)\oplus \mc{A}(\g) \to \mc{W}$ defined in Section~\ref{sect::341}, given by
  \begin{align*}
&q\colon(z,a) \mapsto z+a\theta+I_0,~~\text{ for }z\in Z(\g)~\text{ and }~a\in \mc A(\g)
\end{align*}
is an isomorphism of associative algebras.

By combining the results established in the preceding sections, we obtain the following corollary. This also provides a confirmation of the conjecture in \cite[Conjecture 2]{P14}; see also \cite[Lemma~3.4]{PS16}, where the case of $\mf{osp}(1|2)$ was treated.

\begin{cor}\label{cor:type:osp12n}
Let $\mathfrak{g}=\mathfrak{osp}(1\vert 2n)$. Retaining the notation introduced above, we identify $Z(\mathfrak{g})$ with its image $\pi(Z(\mathfrak{g}))$.  The finite $W$-superalgebra $U(\g,e)$ is a free $Z(\g)$-module with generators $\pi(1), \pi(T), \pi(\theta)$ and $\pi(T\theta)$ satisfying that \begin{align} \label{eq::osp12nsimples}
&\pi(T)^2 \in Z(\g),~~\pi(\theta)^2 ={\frac{1}{2}},~~[\pi(T\theta), \pi(\theta)] = 0.
\end{align} In particular, the irreducible $U(\g,e)$-modules  are parameterized by central characters $\eta: Z(\mathfrak{g}) \rightarrow \mathbb{C}$ and are denoted by $\hat E_{\eta}$ with $\ker(\eta)\hat E_\eta=0$ and determined by: \begin{itemize}
\item If $\eta$ is not strongly typical, then $\hat E_\eta=\text{span}_\C\{v, \pi(\theta)v\}$ is a two-dimensional simple $U(\g,e)$-module of type $\texttt{Q}$ such that $\pi(T)\hat E_\eta =0$;  
\item If $\eta$ is strongly typical, then $\hat E_\eta$ is a two-dimensional simple module of type $\texttt{M}$ obtained by lifting the unique simple module of the quotient algebra $U(\g,e)/(\ker \eta)$ via the canonical projection, where the quotient is isomorphic to the rank-two Clifford superalgebra. 
\end{itemize}
\end{cor}
\begin{proof}
The relations \eqref{eq::osp12nsimples} follow from Theorem   \ref{thm::1}. Furthermore,  the conclusion regarding the classification of irreducible $U(\g,e)$-modules is a direct consequence of Theorem   \ref{thm::UWsimples}. Alternatively, we may observe that   for any character $\eta: Z(\mf{g}) \to \C$, the structure of  the quotient $U(\g,e)/(\ker\eta)$ is determined by the type of typicality of $\eta$. That is, it is isomorphic to a rank-two Clifford superalgebra whenever $\eta$ is strongly typical. In the other case, it is given by the tensor product of a rank-one Clifford superalgebra and a Grassmann superalgebra $\C[x]/(x^2)$, where $x$ is an odd indeterminate. This completes the proof.    
\end{proof}
It is worth pointing out that the classification of irreducible modules for the principal finite $W$-superalgebra {$\mc W$} of $\mf{osp}(1\vert{}2n)$, established earlier in \cite[Section 6]{G24}, can be recovered from the classification of irreducible $U(\g,e)$-modules in Corollary \ref{cor:type:osp12n} via Theorem \ref{thm::UWsimples}.

\end{document}